\documentclass[reqno]{amsart}

\usepackage{amsmath,amsfonts,amssymb,amscd,verbatim,delarray,fullpage}

\usepackage{tikz}
\usepackage{tikz-cd}
\usetikzlibrary{matrix,arrows}
\usepackage{stmaryrd}
\usepackage{amsthm}
\usepackage{url}
\usepackage[polutonikogreek,english]{babel}

\DeclareSymbolFont{bbold}{U}{bbold}{m}{n}
\DeclareSymbolFontAlphabet{\mathbbold}{bbold}

\DeclareMathOperator{\aut}{Aut}

\DeclareMathOperator{\GL}{GL}

\DeclareMathOperator{\SL}{SL}

\chardef\bslash=`\\ 

\begin{document}


\newtheorem{Theorem}{Theorem}[section]

\newtheorem{Definition}[Theorem]{Definition}
\newtheorem{cor}[Theorem]{Corollary}
\newtheorem{goal}[Theorem]{Goal}

\newtheorem{Conjecture}[Theorem]{Conjecture}
\newtheorem{guess}[Theorem]{Guess}

\newtheorem{Example}[Theorem]{Example}
\newtheorem{exercise}[Theorem]{Exercise}
\newtheorem{Question}[Theorem]{Question}
\newtheorem{lemma}[Theorem]{Lemma}
\newtheorem{property}[Theorem]{Property}
\newtheorem{proposition}[Theorem]{Proposition}
\newtheorem{ax}[Theorem]{Axiom}
\newtheorem{claim}[Theorem]{Claim}

\newtheorem{nTheorem}{Surjectivity Theorem}

\theoremstyle{definition}
\newtheorem{problem}[Theorem]{Problem}
\newtheorem{question}[Theorem]{Question}

\newtheorem{remark}[Theorem]{Remark}
\newtheorem{diagram}{Diagram}
\newtheorem{Remark}[Theorem]{Remark}
\newcommand{\diagref}[1]{diagram~\ref{#1}}
\newcommand{\thmref}[1]{Theorem~\ref{#1}}
\newcommand{\secref}[1]{Section~\ref{#1}}
\newcommand{\subsecref}[1]{Subsection~\ref{#1}}
\newcommand{\lemref}[1]{Lemma~\ref{#1}}
\newcommand{\corref}[1]{Corollary~\ref{#1}}
\newcommand{\exampref}[1]{Example~\ref{#1}}
\newcommand{\remarkref}[1]{Remark~\ref{#1}}
\newcommand{\corlref}[1]{Corollary~\ref{#1}}
\newcommand{\claimref}[1]{Claim~\ref{#1}}
\newcommand{\defnref}[1]{Definition~\ref{#1}}
\newcommand{\propref}[1]{Proposition~\ref{#1}}
\newcommand{\prref}[1]{Property~\ref{#1}}
\newcommand{\itemref}[1]{(\ref{#1})}
\newcommand{\ul}[1]{\underline{#1}}


\newcommand{\CE}{\mathcal{E}}
\newcommand{\CG}{\mathcal{G}}\newcommand{\CV}{\mathcal{V}}
\newcommand{\CL}{\mathcal{L}}
\newcommand{\CM}{\mathcal{M}}
\newcommand{\A}{\mathcal{A}}
\newcommand{\CO}{\mathcal{O}}
\newcommand{\B}{\mathcal{B}}
\newcommand{\CS}{\mathcal{S}}
\newcommand{\CX}{\mathcal{X}}
\newcommand{\CY}{\mathcal{Y}}
\newcommand{\CT}{\mathcal{T}}
\newcommand{\CW}{\mathcal{W}}
\newcommand{\CJ}{\mathcal{J}}
\newcommand{\Q}{\mathbb{Q}}
\newcommand{\Z}{\mathbb{Z}}

\newcommand{\st}{\sigma}
\renewcommand{\k}{\varkappa}
\newcommand{\Frac}{\mbox{Frac}}
\newcommand{\XC}{\mathcal{X}}
\newcommand{\wt}{\widetilde}
\newcommand{\wh}{\widehat}
\newcommand{\mk}{\medskip}
\renewcommand{\sectionmark}[1]{}
\renewcommand{\Im}{\operatorname{Im}}
\renewcommand{\Re}{\operatorname{Re}}
\newcommand{\la}{\langle}
\newcommand{\ra}{\rangle}
\newcommand{\LND}{\mbox{LND}}
\newcommand{\Pic}{\mbox{Pic}}
\newcommand{\lnd}{\mbox{lnd}}
\newcommand{\GLND}{\mbox{GLND}}\newcommand{\glnd}{\mbox{glnd}}
\newcommand{\Der}{\mbox{DER}}\newcommand{\DER}{\mbox{DER}}
\renewcommand{\th}{\theta}
\newcommand{\ve}{\varepsilon}
\newcommand{\1}{^{-1}}
\newcommand{\iy}{\infty}
\newcommand{\iintl}{\iint\limits}
\newcommand{\capl}{\operatornamewithlimits{\bigcap}\limits}
\newcommand{\cupl}{\operatornamewithlimits{\bigcup}\limits}
\newcommand{\suml}{\sum\limits}
\newcommand{\ord}{\operatorname{ord}}
\newcommand{\gal}{\operatorname{Gal}}
\newcommand{\bk}{\bigskip}
\newcommand{\fc}{\frac}
\newcommand{\g}{\gamma}
\newcommand{\be}{\beta}
\newcommand{\dl}{\delta}
\newcommand{\Dl}{\Delta}
\newcommand{\lm}{\lambda}
\newcommand{\Lm}{\Lambda}
\newcommand{\om}{\omega}
\newcommand{\ov}{\overline}
\newcommand{\vp}{\varphi}
\newcommand{\kap}{\varkappa}

\newcommand{\Vp}{\Phi}
\newcommand{\Varphi}{\Phi}
\newcommand{\BC}{\mathbb{C}}
\newcommand{\C}{\mathbb{C}}\newcommand{\BP}{\mathbb{P}}
\newcommand{\BQ}{\mathbb {Q}}
\newcommand{\BM}{\mathbb{M}}
\newcommand{\mbh}{\mathbb{H}}
\newcommand{\BR}{\mathbb{R}}\newcommand{\BN}{\mathbb{N}}
\newcommand{\BZ}{\mathbb{Z}}\newcommand{\BF}{\mathbb{F}}
\newcommand{\BA}{\mathbb {A}}
\renewcommand{\Im}{\operatorname{Im}}
\newcommand{\idd}{\operatorname{id}}
\newcommand{\ep}{\epsilon}
\newcommand{\tp}{\tilde\partial}
\newcommand{\doe}{\overset{\text{def}}{=}}
\newcommand{\supp} {\operatorname{supp}}
\newcommand{\loc} {\operatorname{loc}}
\newcommand{\de}{\partial}
\newcommand{\z}{\zeta}
\renewcommand{\a}{\alpha}
\newcommand{\G}{\Gamma}
\newcommand{\der}{\mbox{DER}}

\newcommand{\Spec}{\operatorname{Spec}}
\newcommand{\Sym}{\operatorname{Sym}}
\newcommand{\Aut}{\operatorname{Aut}}

\newcommand{\Idd}{\operatorname{Id}}

\newcommand{\tG}{\widetilde G}

\newcommand{\FX}{\mathfrac {X}}
\newcommand{\FV}{\mathfrac {V}}
\newcommand{\SX}{\mathcal {X}}
\newcommand{\SV}{\mathcal {V}}
\newcommand{\SO}{\mathcal {O}}
\newcommand{\SD}{\mathcal {D}}
\newcommand{\Sr}{\rho}
\newcommand{\SR}{\mathcal {R}}
\newcommand{\cl}{\mathcal{C}}
\newcommand{\ok}{\mathcal{O}_K}

\setcounter{equation}{0} \setcounter{section}{0}

\newcommand{\ds}{\displaystyle}
\newcommand{\gl}{\lambda}
\newcommand{\gL}{\Lambda}
\newcommand{\gge}{\epsilon}
\newcommand{\gG}{\Gamma}
\newcommand{\ga}{\alpha}
\newcommand{\gb}{\beta}
\newcommand{\gd}{\delta}
\newcommand{\gD}{\Delta}
\newcommand{\gs}{\sigma}
\newcommand{\mbq}{\mathbb{Q}}
\newcommand{\mbr}{\mathbb{R}}
\newcommand{\mbz}{\mathbb{Z}}
\newcommand{\mbc}{\mathbb{C}}
\newcommand{\mbn}{\mathbb{N}}
\newcommand{\mbp}{\mathbb{P}}
\newcommand{\mbf}{\mathbb{F}}
\newcommand{\mbe}{\mathbb{E}}
\newcommand{\mf}[1]{\mathfrak{#1}}
\newcommand{\ol}[1]{\overline{#1}}
\newcommand{\mc}[1]{\mathcal{#1}}
\newcommand{\nequiv}{\equiv\hspace{-.07in}/\;}
\newcommand{\bnequiv}{\equiv\hspace{-.13in}/\;}

\title{CM elliptic curves and vertically entangled 2-adic groups}
\date{}
\author{Nathan Jones}
\address{Department of Mathematics, Statistics and Computer Science, University of Illinois at Chicago, 851 S Morgan St, 322
SEO, Chicago, 60607, IL, USA}
\email{ncjones@uic.edu}

\maketitle

\begin{abstract}
Consider the elliptic curve $E$ given by the Weierstrass equation $y^2 = x^3 - 11x - 14$, which has complex multiplication by the order of conductor $2$ inside $\mbz[i]$.  It was recently observed in a paper of Daniels and Lozano-Robledo that, for each $n \geq 2$, $\mbq(\mu_{2^{n+1}}) \subseteq \mbq(E[2^n])$.  In this note, we prove that this (a priori surprising) ``tower of vertical entanglements'' is actually more a feature than a bug:  it holds for \emph{any} elliptic curve $E$ over $\mbq$ with complex multiplication by any order of even discriminant.
\end{abstract}

\section{Main result and proof}

Let $E$ be an elliptic curve over $\mbq$ with complex multiplication by the order $\mc{O}_{K,f} \subseteq \mc{O}_K$ of conductor $f$ inside the imaginary quadratic field $K$.  Since every endomorphism of $E$ defined over $\mbq$ commutes with the action of $\gal(\ol{\mbq}/\mbq)$, it follows that the image of the Galois representation
\[
\rho_{E,m} : \gal(\ol{\mbq}/\mbq) \longrightarrow \aut(E[m]) \simeq \GL_2(\mbz/m\mbz),
\]
(which is defined by letting $\gal(\ol{\mbq}/\mbq)$ act on $E[m]$, the $m$-torsion subgroup of $E$, and fixing a $\mbz/m\mbz$-basis thereof) lies inside a certain subgroup $\mc{N}_{\gd,\phi}(m) \subseteq \GL_2(\mbz/m\mbz)$, which we now specify, following \cite{lozanorobledo}.  First, let us set
\begin{equation*} \label{defoofphiandgd}
\begin{split}
\phi = \phi(\mc{O}_{K,f},m) :=& 
\begin{cases}
0 & \text{ if } \Delta_K f^2 \equiv 0 \bmod{4} \text{ or if $m$ is odd,} \\
f & \text{ if } \Delta_K f^2 \equiv 1 \bmod{4} \text{ and $m$ is even,}
\end{cases} \\
\gd = \gd(\mc{O}_{K,f},m) :=&
\begin{cases}
\gD_K f^2/4 & \text{ if } \Delta_K f^2 \equiv 0 \bmod{4} \text{ or if $m$ is odd,} \\
(\gD_K-1)f^2/4 & \text{ if } \Delta_K f^2 \equiv 1 \bmod{4} \text{ and $m$ is even.}
\end{cases}
\end{split}
\end{equation*}
Next, we define the associated \emph{Cartan subgroup} $\mc{C}_{\gd,\phi}(m)$ by
\begin{equation} \label{defofmcC}
\mc{C}_{\gd,\phi}(m) := \left\{ \begin{pmatrix} a + b\phi & b \\ b\gd & a \end{pmatrix} : a,b \in \mbz/m\mbz, \, a^2 + \phi ab - \gd b^2 \in (\mbz/m\mbz)^\times \right\}.
\end{equation}
Finally, we define $\mc{N}_{\gd,\phi}(m) \subseteq \GL_2(\mbz/m\mbz)$ by
\begin{equation} \label{defofmcN}
\mc{N}_{\gd,\phi}(m) := \left\langle \begin{pmatrix} -1 & 0 \\ \phi & 1 \end{pmatrix}, \, \mc{C}_{\gd,\phi}(m) \right\rangle.
\end{equation}
If $E$ is any elliptic curve over $\mbq$ with CM by $\mc{O}_{K,f}$, then, for an appropriate choice of $\mbz/m\mbz$-basis of $E[m]$, we have $\rho_{E,m}(G_\mbq) \subseteq \mc{N}_{\gd,\phi}(m)$.  For more details, see \cite{lozanorobledo}.

Let $E_{-16}$ be the elliptic curve defined by the Weierstrass equation $y^2 = x^3 - 11x - 14$ (i.e. the elliptic curve with Cremona label {\tt{32a3}}).  The curve $E_{-16}$ has CM by the order $\mc{O} := \mbz + 2i\mbz$ of conductor $2$ inside the field $\mbq(i)$.  Furthermore, as observed in \cite[Theorem 1.5]{danielslozanorobledo}, we have
\begin{equation} \label{verticalentanglement}
n \in \mbn_{\geq 2} \; \Longrightarrow \; \mbq(\zeta_{2^{n+1}}) \subseteq \mbq(E_{-16}[2^n]).
\end{equation}
The authors also observed that the elliptic curves $E_{-4,1}$ and $E_{-4,2}$, given, respectively, by the Weierstrass equations $y^2 = x^2 + x$ and $y^2 = x^3 + 2x$ satisfy
\[
\mbq(E_{-4,1}[2]) = \mbq(\zeta_4) \quad \text{ and } \quad \mbq(\zeta_8) \subseteq \mbq(E_{-4,2}[4]).
\]

The purpose of this note is to show that the two elliptic curves $E_{-4,1}$ and $E_{-4,2}$ also satisfy \eqref{verticalentanglement}.  More generally, we will prove the following theorem.
\begin{Theorem} \label{maintheorem}
Let $E$ be any elliptic curve over $\mbq$ with complex multiplication by an order $\mc{O}_{K,f}$ in an imaginary quadratic field $K$.  Assuming that the discriminant $\gD_K f^2$ of $\mc{O}_{K,f}$ is even, we have that, for each $n \in \mbn_{\geq 2}$, $\mbq(\zeta_{2^{n+1}}) \subseteq \mbq(E[2^n])$.
\end{Theorem}
\begin{proof}
Let us denote by $G(2^n) := \rho_{E,2^n}(G_\mbq) \subseteq \mc{N}_{\gd,\phi}(2^n)$ the mod $2^n$ image associated to $E$.  We will establish that, for each $n \in \mbn_{\geq 2}$, there is a surjective homomorphism $\gd : G(2^n) \twoheadrightarrow (\mbz/2^{n+1}\mbz)^\times$ for which $\det |_{G(2^{n+1})} = \gd \circ \pi$, where $\pi : G(2^{n+1}) \to G(2^n)$ denotes the projection map.  In other words, the following diagram will commute:
\begin{equation} \label{diagramchase}
\begin{tikzcd}
G(2^{n+1}) \rar{\det} \dar{\pi} & (\mbz/2^{n+1}\mbz)^\times \dar \\
G(2^n) \rar{\det} \urar{\gd} & (\mbz/2^n\mbz)^\times.
\end{tikzcd}
\end{equation}
Once established, it will follow that
\begin{equation} \label{fieldcontainment}
\mbq(\zeta_{2^{n+1}}) = \mbq(E[2^{n+1}])^{\ker \det} = \mbq(E[2^{n+1}])^{\pi^{-1}(\ker \gd)} = \mbq(E[2^n])^{\ker \gd} \subseteq \mbq(E[2^n]).
\end{equation}
The key observation is that, since $\gD_K f^2$ is assumed even, it follows that $\phi=0$.  Considering \eqref{defofmcC} and \eqref{defofmcN}, we may then see that
\begin{equation} \label{keyobservation}
n \in \mbn_{\geq 2} \; \Longrightarrow \; \ker \left( \mc{N}_{\gd,\phi}(2^{n+1}) \to \mc{N}_{\gd,\phi}(2^n) \right) \subseteq \SL_2(\mbz/2^{n+1}\mbz).
\end{equation}
(The reason we require $n > 1$ is that otherwise we do not have $\ker \left( \mc{N}_{\gd,\phi}(2^{n+1}) \to \mc{N}_{\gd,\phi}(2^n) \right) \subseteq \mc{C}_{\gd,\phi}(2^{n+1})$, since $\begin{pmatrix} -1 & 0 \\ \phi & 1 \end{pmatrix} \equiv I \bmod{2}$; the consequent of \eqref{keyobservation} is false for $n = 1$.)  We now define the map $\gd$ by
\[
\gd(g) := \det(g'), \quad \text{ where } g' \in \pi^{-1}(g).
\]
By virtue of \eqref{keyobservation}, this is independent of the choice of $g' \in \pi^{-1}(g)$, and thus defines a map $\gd : G(2^n) \to (\mbz/2^{n+1}\mbz)^\times$.  It is surjective since $\det : G(2^{n+1}) \to (\mbz/2^{n+1}\mbz)^\times$ is, and the diagram \eqref{diagramchase} commutes by definition of $\gd$.  Thus, by \eqref{fieldcontainment}, we deduce that
\[
\forall n \in \mbn_{\geq 2}, \quad \mbq(\zeta_{2^{n+1}}) \subseteq \mbq(E[2^n]),
\]
as asserted.
\end{proof}
The proof of Theorem \ref{maintheorem} applies to a more general situation, as follows.  Given an algebraic group $\mc{G}$ and a Galois representation $\rho : \gal(\ol{\mbq}/\mbq) \to \mc{G}(\hat{\mbz})$, let us denote by $\rho_m : \gal(\ol{\mbq}/\mbq) \to \mc{G}(\mbz/m\mbz)$ the composition of $\rho$ with the natural projection map $\mc{G}(\hat{\mbz}) \to \mc{G}(\mbz/m\mbz)$ and define
$
\mbq(E(\rho[m])) := \ol{\mbq}^{\ker \rho_m}.
$
\begin{Definition}
Suppose $\mc{G}$ is any algebraic group that admits a homomorphism $\gd : \mc{G} \to \mathbb{G}_m$ to the multiplicative group.  We say that a Galois representation $\rho : \gal(\ol{\mbq}/\mbq) \to \mc{G}(\hat{\mbz})$ \textbf{\emph{extends the cyclotomic character}} if $\gd_{\hat{\mbz}} \circ \rho : \gal(\ol{\mbq}/\mbq) \to \mc{G}(\hat{\mbz}) \to \hat{\mbz}^\times$ agrees with the cyclotomic character.  For any prime number $p$, we say that $\mc{G}(\mbz_p)$ form a \textbf{\emph{vertically entangled $p$-adic group}} if there exists $n_0 \in \mbn$ so that, for each $n \in \mbn_{\geq n_0}$, we have $\ker\left( \mc{G}(\mbz/p^{n+1}\mbz) \to \mc{G}(\mbz/p^n\mbz) \right) \subseteq \ker \gd_{p^{n+1}}$, where $\gd_{p^{n+1}} : \mc{G}(\mbz/p^{n+1}\mbz) \to (\mbz/p^{n+1}\mbz)^\times$ denotes the group homomorphism associated to $\gd$ on the mod $p^{n+1}$ points of $\mc{G}$.
\end{Definition}
\begin{remark}
Let $\rho : \gal(\ol{\mbq}/\mbq) \to \mc{G}(\hat{\mbz})$ be any Galois representation that extends the cyclotomic character and suppose that, for some prime number $p$, the group $\mc{G}(\mbz_p)$ is a vertically entangled $p$-adic group.  Then the proof of Theorem \ref{maintheorem} shows that, in this more general context, we have
\[
\forall n \in \mbn_{\geq n_0}, \quad \mbq(\mu_{p^{n+1}}) \subseteq \mbq(\rho[p^n]).
\] 
\end{remark}
\section{Acknowledgement}
The author gratefully acknowledges Harris Daniels for bringing the phenomenon \eqref{verticalentanglement} to his attention, and also Ken McMurdy for subsequent stimulating conversations.

\end{document}